\documentclass[reqno]{amsart}
\usepackage{mathrsfs}
\newtheorem{theorem}{Theorem}[section]

\theoremstyle{definition}

\theoremstyle{remark}
\newtheorem{remark}[theorem]{Remark}

\numberwithin{equation}{section}

\newcommand{\di}{\displaystyle}

\begin{document}

\title[Zagier's Conjecture]{On Zagier's Conjecture About the Inverse of a Matrix Related to Double Zeta Values}
\author{Yawen Ma and Lee-Peng Teo}

\address{Department of Mathematics, Xiamen University, Jalan Sunsuria, Bandar Sunsuria, 43900, Sepang, Selangor, Malaysia.}

\email{m1102439194@gmail.com, lpteo@xmu.edu.my}


\date{\today}


\begin{abstract}

We prove a conjecture of Zagier about the inverse of a $(K-1)\times (K-1)$ matrix $A=A_{K}$ using elementary methods. This formula  allows one to  express the the product of single zeta values $\zeta(2r)\zeta(2K+1-2r)$, $1\leq r\leq K-1$, in terms of the double zeta values $\zeta(2r, 2K+1-2r)$, $1\leq r\leq K-1$ and $\zeta(2K+1)$.  
\end{abstract}
\maketitle

\section{Introduction}
 This paper addresses a conjecture of Zagier he put forward in the paper \cite{Zagier2012}. 
Following \cite{Zagier2012}, for positive integers $k_1, \ldots, k_n$ with $k_n\geq 2$, define the multiple zeta value $\zeta(k_1, k_2, \ldots, k_n)$  by
\begin{align}\label{eq1}
\zeta\left(k_1, k_2, \ldots, k_n\right)=&\sum_{1\leq m_1< \ldots< m_n }\frac{1}{m_1^{k_1}\ldots m_n^{k_n}}.
\end{align}$k=k_1+k_2+\ldots+k_n$ is called the weight of this multiple zeta value.

When $n=1$, we have the classical Riemann zeta value
$$\zeta(k)=\sum_{m=1}^{\infty}\frac{1}{m^k}.$$
When $n=2$,  the double sum
\begin{align}\label{eq2}
\zeta(k_1,k_2)=\sum_{m=2}^{\infty}\frac{1}{m^{k_2}}\sum_{j=1}^{m-1}\frac{1}{j^{k_1} }
\end{align}has been considered by Euler.  

Let $H(0)=1$ and define
$$H(n)=\zeta(\underbrace{2, 2, \ldots, 2}_n)\quad\text{for $n\geq 1$}.$$ It is well known that for $n\geq 0$, 
 $$H(n)=\frac{\pi^{2n}}{(2n+1)!}.$$When $a$ and $b$ are nonnegative integers, define 
 $$H(a, b)=\zeta(\underbrace{2,\ldots,2}_a,3, \underbrace{2,\ldots,2}_b).$$ 
In \cite{Zagier2012}, Zagier derived the formula
\begin{align}\label{eq30}
H(a, b)=2\sum_{r=1}^{K}(-1)^r\left[\begin{pmatrix}2r\\2a+2\end{pmatrix}-\left(1-\frac{1}{2^{2r}}\right)\begin{pmatrix}2r\\2b+1\end{pmatrix}\right]H(K-r)\zeta(2r+1),
\end{align}where $K=a+b+1$.
Here for a real number  $n$ and a nonnegative integer $k$, we define the generalized binomial coefficient $\di\begin{pmatrix} n\\k\end{pmatrix}$ by
\begin{align*}
\begin{pmatrix} n\\k\end{pmatrix}=\begin{cases}\di\frac{n(n-1)\ldots (n-k+1)}{k! },\quad & k\geq 1,\\
\quad 1,\quad & k=0.\end{cases}
\end{align*}
In particular, if $n$ is an integer and $n<k$, $\begin{pmatrix} n\\k\end{pmatrix}=0.$

 The formula  \eqref{eq30} expresses   $H(a, b)$ as rational linear combinations of  $H(K-r)\zeta(2r+1)$. It can be used to prove that the following two sets
 \begin{align*}
 \mathscr{B}_1=&\left\{H(a, K-1-a)\,|\, 0\leq a\leq K-1\right\}, \\
 \mathscr{B}_2=&\left\{H(K-r)\zeta(2r+1)\,|\, 1\leq r\leq K\right\} 
 \end{align*} span the same vector space over $\mathbb{Q}$. 
 
 On the other hand, Euler has found that all double zeta values of odd weight can be expressed as rational linear combinations of the Riemann zeta values. In particular, when $1\leq r\leq K-1$,
\begin{align}\label{eq40}
\zeta(2r, 2K+1-2r)=&-\frac{1}{2}\zeta(2K+1) +\sum_{s=1}^{K-1}A_{r,s}\zeta(2s)\zeta(2K+1-2s),
\end{align}where 
\begin{equation}\label{eq11}A_{r,s}=\begin{pmatrix} 2K-2s\\2r-1\end{pmatrix}+\begin{pmatrix} 2K-2s\\2K-2r\end{pmatrix}.\end{equation}
In \cite{Zagier2012}, Zagier used an elementary argument to show that the $(K-1)\times (K-1)$ matrix $A=A_K=[A_{r,s}]$   has nonzero determinant, and thus it is invertible.  Using the fact that both $\zeta (2s)$ and $H(2s)$  are rational multiples of $\pi^{2s}$, this shows that 
the set $$\mathscr{B}_3=\left\{ \zeta(2r, 2K+1-2r)\,|\, 1\leq r\leq K-1\right\}\cup\left\{\zeta(2K+1)\right\}$$ spans the same  $\mathbb{Q}$-vector space as the set $\mathscr{B}_2$. 

In \cite{Zagier2012}, Zagier formulated three conjectures about the matrix $A_K$. The first one is about its determinant, the second one is about its $LU$-decomposition, and the third one is about its inverse. The main objective of this  paper is to prove the third conjecture, which states a pair of conjectural formulas for $A^{-1}$.

Let $P$ and $Q$ be the $(K-1)\times (K-1)$ matrices with entries  
\begin{align*}
P_{s,r}=&\frac{2}{2s-1}\sum_{n=0}^{2K-2s}\begin{pmatrix} 2r-1\\2K-2s-n+1\end{pmatrix} \begin{pmatrix} n+2s-2\\n\end{pmatrix}B_n,\\
Q_{s,r}=&-\frac{2}{2s-1}\sum_{n=0}^{2K-2s}\begin{pmatrix} 2K-2r\\2K-2s-n+1\end{pmatrix} \begin{pmatrix} n+2s-2\\n\end{pmatrix}B_n.
\end{align*}Here $B_n$ are the Bernoulli numbers. 

 Zagier conjectured that both $P$ and $Q$ are the inverse of $A$. This implies that $P$ and $Q$ must be the same matrix. This conjecture was proved by D. Ma in   \cite{Ma2016} using generating functions. In this work, we use a totally different approach.
 
 In Section \ref{equalmatrix}, we will prove that $P=Q$. In Section \ref{conjecture}, we will prove that they indeed give the inverse of $A$. 
 
 \vspace{1cm}
\noindent \textbf{Acknowledgements.}\; This work is supported by the XMUM Research Fund XMUMRF/2018-C2/IMAT/0003.

\section{The equality of the two conjectural formulas}
\label{equalmatrix}
Recall that the Bernoulli numbers $B_n$ are defined by the generating function \cite{Ireland1990}:
$$\frac{t}{e^t-1}=\sum_{n=0}^{\infty}\frac{B_n}{n!}t^n.$$  One can compute $B_n$ recursively by $B_0=1$,  and  
$$B_n=-n!\sum_{k=0}^{n-1}\frac{B_k}{k! (n-k+1)!},\quad\text{for}\; n\geq 1.$$It is well known that $B_1=-1/2$, and for any odd integer $n$ larger than 1, $B_n=0$. 

Let $K$ be an integer greater than or equal to 2, and let $P$ and $Q$ be the $(K-1)\times (K-1)$ matrices with entries defined by
\begin{equation}\label{eq2}\begin{split}
P_{s,r}=&\frac{2}{2s-1}\sum_{n=0}^{2K-2s}\begin{pmatrix} 2r-1\\2K-2s-n+1\end{pmatrix} \begin{pmatrix} n+2s-2\\n\end{pmatrix}B_n,\\
Q_{s,r}=&-\frac{2}{2s-1}\sum_{n=0}^{2K-2s}\begin{pmatrix} 2K-2r\\2K-2s-n+1\end{pmatrix} \begin{pmatrix} n+2s-2\\n\end{pmatrix}B_n.
\end{split}\end{equation}Our goal is to prove that $P_{s,r}=Q_{s,r}$ for all $1\leq r,s\leq K-1$ using a generating function technique that is totally different from that used in \cite{Ma2016}. \footnote{In \cite{Zagier2012}, the summations of $n$ in \eqref{eq2} are taken to be until the term $n=2K-2s+1$. However, since $s\leq K-1$, $2K-2s+1$ is an odd number greater  than 2, and so $B_{2K-2s+1}=0$. Thus the summations  can be taken to be until $n=2K-2s$ only.} We begin with the following theorem which is interesting of its own right.

\begin{theorem}
Let $s$ be a positive integer. Define the function $f_s(t)$ by
$$f_s(t)=\frac{t^{2s-1}}{e^t-1}.$$ If $m$ is a positive integer, we have the following relation that relates the dertivatives of $f_s$ up to  order $m$.
\begin{equation}\label{eq1}\begin{split}
e^tf_s^{(m)}(t)=&\sum_{p=0}^{m} (-1)^{m-p} \begin{pmatrix} m\\p\end{pmatrix} f_s^{(p)}(t)\\&+\sum_{p=0}^{\min\{m,2s-1\}}(-1)^{m-p} \begin{pmatrix} m\\p\end{pmatrix} \frac{(2s-1)!}{(2s-1-p)!}t^{2s-1-p}.
\end{split}\end{equation} 
\end{theorem}
\begin{proof}
By the definition of $f_s(t)$, we have
\begin{align*}
f_s(t)=e^{-t}f_s(t)+t^{2s-1}e^{-t}.
\end{align*}Differentiate both sides $m$ times and apply Leibniz rule, we have
\begin{align*}
f_s^{(m)}(t)=&\sum_{p=0}^m \begin{pmatrix} m\\p\end{pmatrix} f_s^{(p)}(t)\frac{d^{m-p}}{dt^{m-p}}e^{-t}+\sum_{p=0}^{m} \begin{pmatrix} m\\p\end{pmatrix} \frac{d^p}{dt^p}t^{2s-1}\frac{d^{m-p}}{dt^{m-p}}e^{-t}.\end{align*}Since
$$\frac{d^p}{dt^{p}}t^{2s-1}=0\hspace{1cm}\text{if}\; p>2s-1,$$ we find that
\begin{align*}
 f_s^{(m)}(t)=&\sum_{p=0}^{m} (-1)^{m-p} \begin{pmatrix} m\\p\end{pmatrix} e^{-t}f_s^{(p)}(t)\\&+
\sum_{p=0}^{\min\{m,2s-1\}}(-1)^{m-p} \begin{pmatrix} m\\p\end{pmatrix}e^{-t} \frac{(2s-1)!}{(2s-1-p)!}t^{2s-1-p}.
\end{align*}Multiply both sides by $e^t$ give 
\eqref{eq1}.
\end{proof}

Now we can prove the main theorem in this section.
\begin{theorem}
 If $K$ is a positive integer larger than or equal to 2, $r$ and $s$ are positive integers less than $K$, then 
 \begin{equation}\label{eq5}
 \begin{split} &\frac{(2s-2)!}{(2K-2r)!} \sum_{n=0}^{2K-2s}\begin{pmatrix} 2K-2r\\2K-2s-n+1\end{pmatrix} \begin{pmatrix} n+2s-2\\n\end{pmatrix}B_n\\=&-\frac{(2s-2)!}{(2K-2r)!} \sum_{n=0}^{2K-2s}\begin{pmatrix} 2r-1\\2K-2s-n+1\end{pmatrix} \begin{pmatrix} n+2s-2\\n\end{pmatrix}B_n.\end{split}\end{equation} In particular, this implies that the matrices $P$ and $Q$ defined by \eqref{eq2} are equal.
\end{theorem}
\begin{proof}The left hand side of \eqref{eq5} can be rewritten as
\begin{equation}\label{eq6}\begin{split}&
\sum_{n=\max\{0, 2r-2s+1\}}^{2K-2s}\frac{1}{(2K-2s-n+1)!(2s-2r+n-1)!} \frac{(n+2s-2)!}{n!}B_n,\end{split}
\end{equation}and the right hand side of \eqref{eq5} can be rewritten as
\begin{equation}\label{eq4}\begin{split}&
-\sum_{n=\max\{0, 2K-2s-2r+2\}}^{2K-2s}\frac{(2r-1)!}{(2K-2s-n+1)!(2r+2s-2K+n-2)!} \frac{(n+2s-2)!}{(2K-2r)!}\frac{B_n}{n!}.\end{split}
\end{equation}
The proof of \eqref{eq5} is by taking $m=2r-1$ in the equation \eqref{eq1} and  comparing the coefficients of $t^{2K-2r}$ on both sides. Namely, we want to compare the coefficients of $t^{2K-2r}$ on both sides of the equation
\begin{equation}\label{eq3}\begin{split}
e^tf_s^{(2r-1)}(t)=&-\sum_{p=0}^{2r-1} (-1)^{p} \begin{pmatrix} 2r-1\\p\end{pmatrix} f_s^{(p)}(t)\\&-\sum_{p=0}^{\min\{2r-1,2s-1\}}(-1)^{p} \begin{pmatrix} 2r-1\\p\end{pmatrix} \frac{(2s-1)!}{(2s-1-p)!}t^{2s-1-p}.
\end{split}\end{equation} 
Notice that
$$f_s(t)=t^{2s-2}\times\frac{t}{e^t-1}=\sum_{n=0}^{\infty}\frac{B_n}{n!}t^{n+2s-2}.$$Therefore, 
\begin{align*}
f_{s}^{(p)}(t)=\sum_{n=max\{0, p-2s+2\}}^{\infty}B_n\frac{(n+2s-2)!}{(n+2s-2-p)!n!}t^{n+2s-2-p}.
\end{align*}First we consider the coefficient of $t^{2K-2r}$ in the left hand side of \eqref{eq3}, namely, the coefficient of $t^{2K-2r}$ in
\begin{align}\label{eq8}
e^tf_s^{(2r-1)}(t)=\sum_{k=0}^{\infty}\frac{t^k}{k!}\sum_{n=max\{0, 2r-2s+1\}}^{\infty}B_n\frac{(n+2s-2)!}{(n+2s-2r-1)!n!}t^{n+2s-2r-1}.
\end{align}It is given by the expression \eqref{eq6}.

The term on the right hand side of \eqref{eq3} can be written as $T_1+T_2$, where
\begin{align*}
T_1=&-\sum_{p=0}^{2r-1} (-1)^{p} \begin{pmatrix} 2r-1\\p\end{pmatrix} f_s^{(p)}(t)\\
=&-\sum_{p=0}^{2r-1}(-1)^p\frac{(2r-1)!}{p!(2r-1-p)!}\sum_{n=\max\{0,  p-2s+2\}}^{\infty}B_n\frac{(n+2s-2)!}{(n+2s-2-p)!n!}t^{n+2s-2-p}\\
=&-\sum_{n=0}^{\infty}\sum_{p=0}^{\min\{2r-1,n+2s-2\}}(-1)^p\frac{(2r-1)!}{p!(2r-1-p)!}B_n\frac{(n+2s-2)!}{(n+2s-2-p)!n!}t^{n+2s-2-p},\\
T_2=&-\sum_{p=0}^{\min\{2r-1,2s-1\}}(-1)^{p} \frac{(2r-1)!}{p!(2r-1-p)!}\frac{(2s-1)!}{(2s-1-p)!}t^{2s-1-p}.
\end{align*}
 $T_2$ contains a term in $t^{2K-2r}$ if and only if $2s-1\geq 2K-2r$, or equivalently, $r+s\geq K+1$. In this case the coefficient of $t^{2K-2r}$ in $T_2$ is
\begin{align}\label{eq9}
\frac{(2r-1)!}{(2r+2s-2K-1)!(2K-2s)!}\frac{(2s-1)!}{(2K-2r)!}.
\end{align}For the term $T_1$, the coefficient of $t^{2K-2r}$ is
\begin{equation}\label{eq7}\begin{split}
&-\sum_{n=\max\{0, 2K-2s-2r+2\}}^{2K-2s}(-1)^n\frac{(2r-1)!}{(2K-2s-n+1)!(2r+2s-2K+n-2)!}\\&\hspace{8cm}\times \frac{(n+2s-2)!}{(2K-2r)!}\frac{B_n}{n!}. 
\end{split}\end{equation}
When $r+s\leq K$, $2K-2s-2r+2\geq 2$. Hence, the sum over $n$ in \eqref{eq7} does not contain $n=1$ term. Since $B_n=0$ when $n$ is odd and larger than 2, we find that \eqref{eq7} is equal to 
 \eqref{eq4}. Since there are no contribution from $T_2$ to the term $t^{2K-2r}$ when $r+s\leq K$, this proves that when $r+s\leq K$, the coefficient of $t^{2K-2r}$ in the right hand side of 
 \eqref{eq3} is \eqref{eq4}.
 
 When $r+s\geq K+1$, there is a term with $n=1$ in \eqref{eq7}. Using the fact that $\di B_1=-\frac{1}{2}$, we find that this term is given by
 \begin{align*}
-\frac{1}{2}\frac{(2r-1)!}{(2r+2s-2K-1)!(2K-2s)!}\frac{(2s-1)!}{(2K-2r)!},
\end{align*}which is $-1/2$ of the term \eqref{eq9}. Summing the coefficients of $t^{2K-2r}$ from $T_1$ and $T_2$, we find that the sum is equal to \eqref{eq4}. Therefore, when $r+s\geq K+1$, the coefficient of $t^{2K-2r}$ in the right hand side of 
 \eqref{eq3} is \eqref{eq4}.

Thus, we have shown that the coefficient of $t^{2K-2r}$ in the left hand side of \eqref{eq3} is \eqref{eq6}, and the coefficient of $t^{2K-2r}$ in the right hand side of \eqref{eq3} is \eqref{eq4}, this completes the proof of the theorem.

\end{proof}
As we mentioned before, this theorem has been proved in \cite{Ma2016}  using a totally different method, with the help of the Carlitz's Bernoulli number identity \cite{Carlitz1971, Prodinger2014}. Our proof uses directly the generating function of the Bernoulli numbers. 

\begin{remark} Carlitz's identity says that for any nonnegative integers $m$ and $n$,
\begin{equation}\label{eq10} (-1)^m\sum_{k=0}^m\begin{pmatrix} m\\k\end{pmatrix} B_{n+k}=(-1)^n\sum_{k=0}^n\begin{pmatrix} n\\k\end{pmatrix} B_{m+k}.\end{equation}
Prodinger \cite{Prodinger2014} gave a short  proof using an exponential generating function of two variables.  Here we show that  this identity can be derived directly from \eqref{eq1} by setting $s=1$. Namely, we consider the generating function  of the Bernoulli numbers
$$f(t)=\frac{t}{e^t-1}.$$Since \eqref{eq9} is symmetric in $m$ and $n$, it is sufficient to consider the  case $0\leq m< n$. In this case, $n\geq 1$. 
Equation \eqref{eq1} says that
$$e^tf^{(m)}(t)=(-1)^m t+(-1)^{m-1}m+\sum_{k=0}^{m}(-1)^{m-k}\begin{pmatrix} m\\k\end{pmatrix}f^{(k)}(t).$$
This gives
\begin{align*}
\sum_{l=0}^{\infty}\frac{t^l}{l!}\sum_{k=m}^{\infty}\frac{B_k}{(k-m)!}t^{k-m}=(-1)^mt+(-1)^{m-1}m+\sum_{k=0}^m (-1)^{m-k}\begin{pmatrix} m\\k\end{pmatrix}\sum_{l=k}^{\infty}\frac{B_l}{(l-k)!}t^{l-k}.
\end{align*}Compare the coefficients of $t^n$ on both sides, we have
\begin{align*}
\sum_{k=m}^{m+n}\frac{1}{(k-m)!(m+n-k)!}B_k=&(-1)^m\delta_{n,1}+\sum_{k=0}^{m}(-1)^{m-k}\begin{pmatrix} m\\k\end{pmatrix} \frac{B_{n+k}}{n!}.
\end{align*}If $n=1$, the last sum contains the term $(-1)^m B_1$, which can be combined with $(-1)^m\delta_{n,1}$ to yield
$(-1)^{m-1}B_1.$ Since $B_k=0$ when $k$ is odd and greater than 2, we find that 
 \begin{align*}
\sum_{k=m}^{m+n}\frac{n!}{(k-m)!(m+n-k)!}B_k=& \sum_{k=0}^{m}(-1)^{m-n}\begin{pmatrix} m\\k\end{pmatrix} B_{n+k}.
\end{align*}Multiplying $(-1)^n$ on both sides and shifting the summation variables on the left hand side, we obtain
\begin{align*}
(-1)^n\sum_{k=0}^{n}\begin{pmatrix} n\\k\end{pmatrix}B_{m+k}=& (-1)^m\sum_{k=0}^{m} \begin{pmatrix} m\\k\end{pmatrix} B_{n+k},
\end{align*}
which is the Carlitz identity.

\end{remark}
 
\section{The proof of the conjecture}\label{conjecture}
In this section, we prove that the inverse of the matrix $A=A_K$ is the matrix $P$.
\begin{theorem}
[Zagier's Conjecture]~\\
If $K$ is an integer larger than 1, $A_K$ is the $(K-1)\times (K-1)$ matrix defined by \begin{equation}\label{eq12}A_{r,s}=\begin{pmatrix} 2K-2s\\2r-1\end{pmatrix}+\begin{pmatrix} 2K-2s\\2K-2r\end{pmatrix},\end{equation} then the inverse of $A$ is the matrix $P$ defined by  one of the following two formulas that are equal.
\begin{equation}\label{eq13}\begin{split}
P_{s,r}=&\frac{2}{2s-1}\sum_{n=0}^{2K-2s}\begin{pmatrix} 2r-1\\2K-2s-n+1\end{pmatrix} \begin{pmatrix} n+2s-2\\n\end{pmatrix}B_n,\\
=&-\frac{2}{2s-1}\sum_{n=0}^{2K-2s}\begin{pmatrix} 2K-2r\\2K-2s-n+1\end{pmatrix} \begin{pmatrix} n+2s-2\\n\end{pmatrix}B_n.
\end{split}\end{equation}
\end{theorem} 
\begin{proof}
Define the $(K-1)\times (K-1)$ matrices $B$ and $C$ by
\begin{align*}
B_{r,s}=\begin{pmatrix} 2K-2s\\2r-1\end{pmatrix},\hspace{1cm}C_{r,s}=\begin{pmatrix} 2K-2s\\2K-2r\end{pmatrix},
\end{align*}so that $A=B+C$. Notice that $B_{r,s}=0$ if $r+s> K$, and $C_{r,s}=0$ if $r<s$. 

The strategy of proof is to  show that the matrix $PA=PB+PC$ is indeed the identity matrix, by showing that if $s$ and $s'$ are positive integers less than $K$, then
\begin{align}\label{eq17}
(PB)_{s,s'}+(PC)_{s,s'}=\begin{cases} 1,\quad & s=s',\\0,\quad & s\neq s'.\end{cases}
\end{align}We will use the following two elementary identities of combination numbers.
If $n\geq 1$, then
\begin{equation}\label{eq14}\begin{split}
\sum_{k=0}^{\lfloor n/2\rfloor}\begin{pmatrix} n\\2k\end{pmatrix}=&2^{n-1},\\
\sum_{k=0}^{\lfloor (n-1)/2\rfloor}\begin{pmatrix} n\\2k+1\end{pmatrix}=&2^{n-1}.
\end{split}\end{equation}First we compute $(PB)_{s,s'}$ using the first formula in \eqref{eq13} for $P_{s,r}$. The cases where $s\leq s'$ and $s>s'$ are considered separately.

If $s\leq s'$, then for $r\leq K-s'$, we have $r\leq K-s$. Hence, $2K-2s-2r+2\geq 2$. Therefore, 
\begin{align*}
&(PB)_{s,s'}\\
=&\frac{2}{2s-1}\sum_{r=1}^{K-s'}\begin{pmatrix} 2K-2s'\\2r-1\end{pmatrix}\sum_{n=  2K-2s-2r+2}^{2K-2s}\begin{pmatrix} 2r-1\\2K-2s-n+1\end{pmatrix} \begin{pmatrix} n+2s-2\\n\end{pmatrix}B_n\\
=&\frac{2}{2s-1}\sum_{n=2s'-2s+2}^{2K-2s}\sum_{r=K-s-\frac{n}{2}+1}^{K-s'}\begin{pmatrix} 2K-2s'\\2r-1\end{pmatrix}\begin{pmatrix} 2r-1\\2K-2s-n+1\end{pmatrix} \begin{pmatrix} n+2s-2\\n\end{pmatrix}B_n. \end{align*}
Notice that the summation over $n$ only contains terms with $n$ even since $2s'-2s+2\geq 2$.

 It can be easily verified that 
\begin{align*}
\begin{pmatrix} 2K-2s'\\2r-1\end{pmatrix}\begin{pmatrix} 2r-1\\2K-2s-n+1\end{pmatrix} =& \begin{pmatrix} 2K-2s'\\2s-2s'+n-1\end{pmatrix}\begin{pmatrix} 2s-2s'+n-1\\2r+2s-2K+n-2\end{pmatrix}.
\end{align*}Therefore,
\begin{align*}
(PB)_{s,s'}=&\frac{2}{2s-1}\sum_{n=2s'-2s+2}^{2K-2s}   \begin{pmatrix} 2K-2s'\\2s-2s'+n-1\end{pmatrix}\begin{pmatrix} n+2s-2\\n\end{pmatrix}B_n  \\&\hspace{3cm} \times\sum_{r=K-s-\frac{n}{2}+1}^{K-s'}\begin{pmatrix} 2s-2s'+n-1\\2r+2s-2K+n-2\end{pmatrix}.
\end{align*} For $n\geq 2s'-2s+2$, we have $2s-2s'+n-1\geq 1$. The first formula  in \eqref{eq14} implies that
\begin{equation}\label{eq22}\begin{split}
\sum_{r=K-s-\frac{n}{2}+1}^{K-s'}\begin{pmatrix} 2s-2s'+n-1\\2r+2s-2K+n-2\end{pmatrix}
=&\sum_{r=0}^{s-s'+n/2-1}\begin{pmatrix} 2s-2s'+n-1\\2r\end{pmatrix}\\=&2^{2s-2s'+n-2}.\end{split}\end{equation}
This shows that when $s\leq s'$, 
\begin{align}\label{eq25}
(PB)_{s,s'}=&\frac{2}{2s-1}\sum_{n=2s'-2s+2}^{2K-2s}  \begin{pmatrix} 2K-2s'\\2s-2s'+n-1\end{pmatrix}\begin{pmatrix} n+2s-2\\n\end{pmatrix}2^{2s-2s'+n-2}B_n.
\end{align}
When $s>s'$,  
\begin{align*}
(PB)_{s,s'}=&\frac{2}{2s-1} \sum_{n=0}^{2K-2s}\sum_{K-s-\frac{n}{2}+1\leq r\leq K-s'}\begin{pmatrix} 2K-2s'\\2r-1\end{pmatrix}\begin{pmatrix} 2r-1\\2K-2s-n+1\end{pmatrix} \begin{pmatrix} n+2s-2\\n\end{pmatrix}B_n.
\end{align*}Splitting out the   $n=1$ term, we have
\begin{align*}
(PB)_{s,s'}=&\frac{2}{2s-1}\Biggl\{\sum_{\substack{0\leq n\leq 2K-2s\\n\,\text{is even}}}\begin{pmatrix} 2K-2s'\\2s-2s'+n-1\end{pmatrix}\begin{pmatrix} n+2s-2\\n\end{pmatrix}B_n\\&\hspace{4cm}\times\sum_{r=K-s-\frac{n}{2}+1}^{K-s'} \begin{pmatrix} 2s-2s'+n-1\\2r+2s-2K+n-2\end{pmatrix} \\
 &\hspace{1cm}+\begin{pmatrix} 2K-2s'\\2s-2s'\end{pmatrix}\begin{pmatrix}2s-1\\1\end{pmatrix}B_1\sum_{r=K-s +1}^{K-s'} \begin{pmatrix} 2s-2s'\\2r+2s-2K-1\end{pmatrix} \Biggr\}.\end{align*}Notice that the second formula in \eqref{eq14} give
$$\sum_{r=K-s +1}^{K-s'}  \begin{pmatrix} 2s-2s'\\2r+2s-2K-1\end{pmatrix} =\sum_{r=0}^{s-s'-1} \begin{pmatrix} 2s-2s'\\2r+1\end{pmatrix} =2^{2s-2s'-1}.$$Together with \eqref{eq22}, we find that when $s>s'$,  
\begin{align*}
(PB)_{s,s'}=&\frac{2}{2s-1} \sum_{n=0}^{2K-2s}  \begin{pmatrix} 2K-2s'\\2s-2s'+n-1\end{pmatrix}  \begin{pmatrix} n+2s-2\\n\end{pmatrix}2^{2s-2s'+n-2}B_n.
\end{align*} 

Next we compute $(PC)_{s,s'}$. Using the second expression in \eqref{eq13} for $P_{s,r}$, we have
\begin{align*}
&(PC)_{s,s'}\\
=&-\frac{2}{2s-1}\sum_{r=s'}^{K-1}\begin{pmatrix} 2K-2s'\\2K-2r\end{pmatrix} \sum_{n=\max\{0,2r-2s+1\}}^{2K-2s}\begin{pmatrix} 2K-2r\\2K-2s-n+1\end{pmatrix} \begin{pmatrix} n+2s-2\\n\end{pmatrix}B_n.
\end{align*}Again, it is easy to verify that
\begin{align*}
\begin{pmatrix} 2K-2s'\\2K-2r\end{pmatrix} \begin{pmatrix} 2K-2r\\2K-2s-n+1\end{pmatrix} =\begin{pmatrix} 2K-2s'\\2s-2s'+n-1\end{pmatrix}\begin{pmatrix} 2s-2s'+n-1\\2r-2s'\end{pmatrix}. 
\end{align*}
Now we discuss the cases $s<s'$, $s=s'$ and $s>s'$ separately. 

When $s=s'$,
\begin{align*}
(PC)_{s,s}=&-\frac{2}{2s-1}\sum_{r=s}^{K-1}\begin{pmatrix} 2K-2s\\2K-2r\end{pmatrix} \sum_{n= 2r-2s+1 }^{2K-2s}\begin{pmatrix} 2K-2r\\2K-2s-n+1\end{pmatrix} \begin{pmatrix} n+2s-2\\n\end{pmatrix}B_n.
\end{align*}In this case, we have a $n=1$ term when $r=s$. When $r>s$, summation over $n\geq 2r-2s+1$ is the same as summation over $2r-2s+2$.  The term with $r=s$ and $n=1$ contribute the term 1. Therefore,
\begin{align*}
(PC)_{s,s}=& 1-\frac{2}{2s-1}\sum_{r=s}^{K-1}\begin{pmatrix} 2K-2s\\2K-2r\end{pmatrix} \sum_{n=2r-2s+2}^{2K-2s}\begin{pmatrix} 2K-2r\\2K-2s-n+1\end{pmatrix} \begin{pmatrix} n+2s-2\\n\end{pmatrix}B_n\\
=&1-\frac{2}{2s-1}\sum_{n=2}^{2K-2s}\sum_{r=s}^{s+n/2-1}\begin{pmatrix} 2K-2s\\n-1\end{pmatrix}\begin{pmatrix} n-1\\2r-2s\end{pmatrix} \begin{pmatrix} n+2s-2\\n\end{pmatrix}B_n\\
=&1-\frac{2}{2s-1}\sum_{n=2}^{2K-2s}\begin{pmatrix} 2K-2s\\n-1\end{pmatrix} \begin{pmatrix} n+2s-2\\n\end{pmatrix}2^{n-2}B_n.
\end{align*}The last equality follows from the first equation in \eqref{eq14}. Compare to the   $s=s'$ case in \eqref{eq25}   show that
$$(PB+PC)_{s,s}=1.$$

Next we consider the case $s<s'$. In this case, if $r\geq s'$, then $r>s$ and hence $2r-2s+1\geq 3$. Therefore,
\begin{align*}
&(PC)_{s,s'}\\=&-\frac{2}{2s-1}\sum_{r=s'}^{K-1}  \sum_{n=2r-2s+2}^{2K-2s}\begin{pmatrix} 2K-2s'\\2s-2s'+n-1\end{pmatrix}\begin{pmatrix} 2s-2s'+n-1\\2r-2s'\end{pmatrix}  \begin{pmatrix} n+2s-2\\n\end{pmatrix}B_n\\
=&-\frac{2}{2s-1}\sum_{n=2s'-2s+2}^{2K-2s}\begin{pmatrix} 2K-2s'\\2s-2s'+n-1\end{pmatrix}\begin{pmatrix} n+2s-2\\n\end{pmatrix}B_n\sum_{r=s'}^{s+n/2-1} \begin{pmatrix} 2s-2s'+n-1\\2r-2s'\end{pmatrix} \\
=&-\frac{2}{2s-1}\sum_{n=2s'-2s+2}^{2K-2s}\begin{pmatrix} 2K-2s'\\2s-2s'+n-1\end{pmatrix}\begin{pmatrix} n+2s-2\\n\end{pmatrix}2^{2s-2s'+n-2}B_n\\
=&-(PB)_{s,s'}.
\end{align*}Finally, we consider the case $s>s'$.  In this case
 \begin{align*}
&(PC)_{s,s'}\\=&-\frac{2}{2s-1} \sum_{n=0}^{2K-2s}\sum_{s'\leq r\leq s+n/2-1} \begin{pmatrix} 2K-2s'\\2K-2r\end{pmatrix} \begin{pmatrix} 2K-2r\\2K-2s-n+1\end{pmatrix} \begin{pmatrix} n+2s-2\\n\end{pmatrix}B_n.
\end{align*}Splitting out the   $n=1$ term, we have
\begin{align*}
&(PC)_{s,s'}\\ 
=&-\frac{2}{2s-1}\Biggl\{ \sum_{\substack{0\leq n\leq 2K-2s\\n\,\text{is even}}} \begin{pmatrix} 2K-2s'\\2s-2s'+n-1\end{pmatrix} \begin{pmatrix} n+2s-2\\n\end{pmatrix}B_n\sum_{r=s'}^{s+n/2-1}\begin{pmatrix} 2s-2s'+n-1\\2r-2s'\end{pmatrix} \\&  \hspace{2cm}+\begin{pmatrix} 2K-2s'\\2s-2s'\end{pmatrix}\begin{pmatrix} 2s-1\\1\end{pmatrix}B_1\sum_{r=s'}^s\begin{pmatrix} 2s-2s'\\2r-2s'\end{pmatrix}\Biggr\} \\
=&-\frac{2}{2s-1} \sum_{n=0}^{2K-2s} \begin{pmatrix} 2K-2s'\\2s-2s'+n-1\end{pmatrix} \begin{pmatrix} n+2s-2\\n\end{pmatrix}2^{2s-2s'+n-2}B_n\\
=&-(PB)_{s,s'}.
\end{align*}This completes the proof of \eqref{eq17}, and so the assertion of the theorem is proved. 
\end{proof}

\bibliographystyle{amsplain}

\end{document}